\documentclass[11pt]{amsart}
\usepackage[a4paper, margin=1in]{geometry}
\usepackage{amsfonts}
\usepackage{amssymb}
\usepackage{amsmath}
\usepackage{amsthm}
\usepackage{mathtools}
\usepackage{hyperref}
\usepackage{cleveref}
\usepackage{stmaryrd}
\usepackage{tikz-cd}
\usepackage{enumitem}
\usepackage{multirow}
\usepackage{bbm}
\usepackage{comment}
\usepackage{float}
\usepackage{url}
\usepackage[en-GB]{datetime2}
\DTMlangsetup[en-GB]{ord=omit}

\title[The hypergeom. datum $\big((\frac{1}{2},\frac{1}{6},\frac{5}{6}),(1,1)\big)$ and sym. squares of elliptic curves]{A note on the hypergeometric datum $\big((\frac{1}{2},\frac{1}{6},\frac{5}{6}),(1,1)\big)$ and symmetric squares of elliptic curves}
\author{Pengcheng Zhang}
\address{Max Planck Institute for Mathematics, Vivatsgasse 7, 53111 Bonn, Germany}
\email{pzhang@mpim-bonn.mpg.de}
\date{29 December 2025}

\newtheorem{theorem}{Theorem}[section]

\newtheorem{lemma}[theorem]{Lemma}
\newtheorem{proposition}[theorem]{Proposition}
\newtheorem{corollary}[theorem]{Corollary}

\theoremstyle{remark}
\newtheorem{remark}[theorem]{Remark}

\def\MR#1{}

\newcommand\fr[1]{\mathfrak{#1}}
\newcommand\ov[1]{\overline{#1}}

\newcommand\bo[1]{\boldsymbol{#1}}

\newcommand{\Q}{\mathbb{Q}}
\newcommand{\Z}{\mathbb{Z}}

\newcommand{\C}{\mathbb{C}}

\newcommand{\F}{\mathbb{F}}

\newcommand{\bslash}{\backslash}
\newcommand{\simarrow}{\xrightarrow{\;\sim\;}}
\newcommand{\eq}{\;=\;}
\newcommand{\deq}{\;:=\;}

\def\thin{{\hskip 1pt}}
\newcommand{\hash}{\texttt{\#}}

\newcommand{\N}{\mathrm{N}}
\newcommand{\GL}{\mathrm{GL}}

\newcommand{\tr}{\mathrm{tr}}
\newcommand{\Gal}{\mathrm{Gal}}

\newcommand{\Sym}{\mathrm{Sym}}

\newcommand{\et}{\mathrm{\acute{e}t}}

\newcommand{\disc}{\mathrm{disc}}

\newcommand{\Frob}{\mathrm{Frob}}
\newcommand{\Trun}{\mathrm{Trun}}

\newmuskip\pFqskip
\mathchardef\pFcomma=\mathcode`, 

\newcommand*\pFq[5]{%
  \begingroup
  \begingroup\lccode`~=`,
    \lowercase{\endgroup\def~}{\mkern\pFqskip}%
  \mathcode`,=\string"8000
  {}_{#1}F_{#2}\biggl[\genfrac..{0pt}{}{#3}{#4};#5\biggr]%
  \endgroup
}

\newcommand*\smallpFq[5]{%
  \begingroup
  \begingroup\lccode`~=`,
    \lowercase{\endgroup\def~}{\mkern\pFqskip}%
  \mathcode`,=\string"8000
  {}_{#1}F_{#2}\bigl[\genfrac..{0pt}{}{#3}{#4};#5\bigr]%
  \endgroup
}

\begin{document}

\maketitle

\begin{abstract}
This is an expository note on a mod $p$ congruence relating the truncated hypergeometric sums associated to $\big((\frac{1}{2},\frac{1}{6},\frac{5}{6}),(1,1)\big)$ to symmetric squares of elliptic curves.
\end{abstract}

\setcounter{tocdepth}{1}
\tableofcontents

\section{Introduction}

In this short note, we will prove a mod $p$ congruence relating the truncated hypergeometric sums associated to $\big((\frac{1}{2},\frac{1}{6},\frac{5}{6}),(1,1)\big)$ to symmetric squares of elliptic curves. Though this congruence should follow from previous works on hypergeometric functions, the result does not appear very well-documented (see \Cref{reference-remark}). This note thus serves mostly as an exposition which aims at proving the congruence in a self-contained way.

We first make some definitions. Given a hypergeometric datum $(\bo{\alpha},\bo{\beta})$ with
\begin{align*}
    \bo{\alpha}&\eq(\alpha_1,\cdots,\alpha_n)\in\C^n \\
    \bo{\beta}&\eq(\beta_1,\cdots,\beta_{n-1})\in\C^{n-1},
\end{align*}
the associated hypergeometric function is defined as
\begin{align*}
    \pFq{n}{n-1}{\bo{\alpha}}{\bo{\beta}}{z}\deq\sum_{m=0}^\infty\frac{(\alpha_1)_m(\alpha_2)_m\cdots(\alpha_n)_m}{(\beta_1)_m(\beta_2)_m\cdots(\beta_{n-1})_m}\cdot \frac{z^m}{m!}
\end{align*}
and the truncated hypergeometric sum is defined as
\begin{align*}
    \pFq{n}{n-1}{\bo{\alpha}}{\bo{\beta}}{z}_{r}\deq\sum_{m=0}^r\frac{(\alpha_1)_m(\alpha_2)_m\cdots(\alpha_n)_m}{(\beta_1)_m(\beta_2)_m\cdots(\beta_{n-1})_m}\cdot \frac{z^m}{m!},
\end{align*}
where $(a)_m:=\prod_{j=0}^{m-1}(a+j)$. In general, for a power series

Let $L$ be a number field and let $\fr{p}$ be a prime of $L$. Let $\N(\fr{p})$ denote the ideal norm of $\fr{p}$, $v_\fr{p}$ denote the $\fr{p}$-adic valuation on $L$, and $\F_\fr{p}$ denote the residue field at $\fr{p}$. For $a\in L$ with $v_\fr{p}(a)\geq 0$, define
\begin{align*}
    \bigg(\frac{a}{\fr{p}}\bigg)\deq\begin{cases}1&\text{ if $a$ is a square in $\F_\fr{p}^\times$}, \\ -1&\text{ if $a$ is not a square in $\F_\fr{p}^\times$}, \\ 0&\text{ if $a=0$ in $\F_\fr{p}$}. \end{cases}
\end{align*}
When $\fr{p}\nmid 2$, it is easy to check that
\begin{align*}
    \bigg(\frac{a}{\fr{p}}\bigg)\eq a^{\frac{\N(\fr{p})-1}{2}}
\end{align*}
as elements in $\F_\fr{p}$.

Now, for an elliptic curve $E/L$ and a good prime $\fr{p}$ of $E$, define
\begin{align*}
    a_\fr{p}(E)\deq\N(\fr{p})+1-|E(\F_\fr{p})|.
\end{align*}

\begin{theorem}
\label{congruence-theorem}
Let $L$ be a number field, $E/L$ be an elliptic curve with $j(E)\notin\{0,1728\}$, and $\fr{p}$ be a prime of $L$ that is a good prime of $E$. Suppose that $\fr{p}\nmid 6$ and $v_\fr{p}(j(E))=0=v_\fr{p}((j(E)-1728))$. Then,
\begin{align*}
    a_\fr{p}(E)^2\;\equiv\;\bigg(\frac{1-1728/j(E)}{\fr{p}}\bigg)\cdot\pFq{3}{2}{\frac{1}{2},\frac{1}{6},\frac{5}{6}}{1,1}{\frac{1728}{j(E)}}_{\N(\fr{p})-1}\pmod{\fr{p}}.
\end{align*}
\end{theorem}

\begin{remark}
\label{reference-remark}
We make several remarks on the relating works and references.
\begin{enumerate}[label=(\arabic*)]
    \item An analogue of this congruence (at least over $\Q$) is known for hypergeometric character sums by \cite[Theorem~1]{hoffman-li-long-tu} where one obtains an equality between $a_p(E)^2-p$, i.e., the trace of Frobenius at $p$ on $\Sym^2 H_\et^1(E,\Q_\ell)$, and the product of $(\frac{1-1728/j(E)}{p})$ and the hypergeometric character sum associated to $\big((\frac{1}{2},\frac{1}{6},\frac{5}{6}),(1,1)\big)$. Here one should be able to relate hypergeometric character sums to truncated hypergeometric sums (see \cite[Theorem~4]{long-numeric}, \cite[Lemma~5.9]{allen-grove-long-tu}, \cite[Section~4.5]{hoffman-li-long-tu}, and \cite[Section~7.2]{babei-roy-swisher-tobin-tu}).
    \item In the CM case, it is expected that the congruence should hold modulo $\fr{p}^2$, namely, the phenomenon of \emph{supercongruences} arises, if one replaces $a_\fr{p}(E)^2$ by $a_\fr{p}(E)^2-2\N(\fr{p})$. Here we refer to \cite[Theorem~4]{mortenson}, \cite[Theorem~1.3]{sun-zw}, \cite[Section~5 and Section~6]{babei-roy-swisher-tobin-tu}, and \cite[Theorem~1.26]{beukers-supercong} for relevant results on supercongruences over $\Q$, we refer to \cite[Observation~4]{zudilin-hypergeom-cy} for some explicit observations on the supercongruences in the ${}_3F_2$ case, and we refer to \cite{rodriguez-villegas-hypergeom-cy,rrv-supercong} for general discussions on hypergeometric supercongruences.
    \item The result by Asakura \cite[Corollary~0.2]{asakura} may be viewed as some analogue of this congruence from the viewpoint of Galois representations.
\end{enumerate}
\end{remark}

\subsection*{Acknowledgements}

The author would like to thank Ling Long, Danylo Radchenko, Fernando Rodriguez Villegas, Fang-Ting Tu, and Wadim Zudilin for useful discussions on hypergeometric functions. Moreover, the author is especially grateful to Chi Hong Chow for his vital idea on proving \Cref{2f1-z-to-(1-z)}.

\section{Preliminaries on elliptic curves}

We start by recalling some well-known preliminary results on elliptic curves for completeness.

\begin{lemma}
\label{isom-over-quad}
Let $L$ be a number field and let 
\begin{align*}
    E_1:\;y^2&\eq x^3+A_1x+B_1 \\
    E_2:\;y^2&\eq x^3+A_2x+B_2
\end{align*}
be two elliptic curves over $L$. Suppose that $j(E_1)=j(E_2)\notin\{0,1728\}$. Let $\gamma=\frac{A_1/B_1}{A_2/B_2}$. Then, $E_1$ and $E_2$ are isomorphic over $L(\sqrt{\gamma})$, and they are isomorphic over $L$ if and only if $\gamma\in (L^\times)^2$.
\end{lemma}
\begin{proof}
This is exactly \cite[Chapter~X, Exercise~10.21, p.~360]{silverman-elliptic-curve}. We will give a proof of the first part, which is essentially the same proof of \cite[Chapter~III, Proposition~1.4~(b), p.~47]{silverman-elliptic-curve}.

Since $j(E_1)=j(E_2)$, we have
\begin{align*}
    \frac{(4A_1)^3}{4A_1^3+27B_1^2}\eq\frac{(4A_2)^3}{4A_2^3+27B_2^2},
\end{align*}
and hence
\begin{align*}
    (A_1/A_2)^3\eq(B_1/B_2)^2.
\end{align*}
Here all $A_1,A_2,B_1,B_2$ are nonzero since $j(E_1)=j(E_2)\notin\{0,1728\}$. Thus, over $L(\sqrt{\gamma})$,
\begin{align*}
    E_1&\simarrow E_2 \\
    (x,y)&\longmapsto \big((A_2/A_1)^{1/2}x,(B_2/B_1)^{1/2}y\big).
\end{align*}
\end{proof}

\begin{lemma}
\label{same-jinv-square-equal}
Let $L$ be a number field and $E_1,E_2$ be two elliptic curves over $L$. Suppose that $j(E_1)=j(E_2)\notin\{0,1728\}$. Then, $a_\fr{p}(E_1)^2=a_\fr{p}(E_2)^2$ for all good primes $\fr{p}$ of both $E_1$ and~$E_2$.
\end{lemma}
\begin{proof}
Fix a good prime $\fr{p}$ of both $E_1$ and $E_2$ and fix a rational prime $\ell$ with $\fr{p}\nmid\ell$. Let $\rho_i:\Gal(\ov{L}/L)\rightarrow\GL_2(\Q_\ell)$ be the $\ell$-adic Galois representation attached to $E_i$. By \Cref{isom-over-quad}, there exists a quadratic extension $L'/L$ such that $\rho_1|_{\Gal(\ov{L}/L')}\cong\rho_2|_{\Gal(\ov{L}/L')}$. As $\Frob_\fr{p}^2\in\Gal(\ov{L}/L')$ since $L'/L$ is quadratic, we thus have
\begin{align*}
    \tr(\rho_1(\Frob_\fr{p}))^2\eq\tr(\rho_1(\Frob_\fr{p}^2))+2\det(\rho_1(\Frob_\fr{p}))&\eq\tr(\rho_2(\Frob_\fr{p}^2))+2\det(\rho_2(\Frob_\fr{p})) \\
    &\eq\tr(\rho_2(\Frob_\fr{p}))^2,
\end{align*}
where we also use that $\det(\rho_i(\Frob_\fr{p}))=\N(\fr{p})$. The result now follows from the fact that $a_\fr{p}(E_i)=\tr(\rho_i(\Frob_\fr{p}))$.
\end{proof}

\section{Reinterpretation of $a_\fr{p}(E)^2$}

Now, we will do a first reinterpretation of $a_\fr{p}(E)^2$. Since $j(E)\notin\{0,1728\}$, by \Cref{same-jinv-square-equal}, we may replace $E/L$ by a different elliptic curve with the same $j$-invariant. Let
\begin{align*}
    j_0:=j(E)\in L-\{0,1728\}\quad\text{ and }\quad z_0=1-1728/j_0\thin(\in L-\{0,1\}).
\end{align*}
Consider
\begin{alignat*}{2}
    E_0/L:&\hspace{1.2cm}y^2&&\eq x^3-\frac{1}{48z_0^3}x+\frac{1}{864z_0^4} \\
    E_1/L(\sqrt{z_0}):&\hspace{0.2cm}y^2+xy&&\eq x^3-\frac{1-\sqrt{z_0}}{864}.
\end{alignat*}
It is easy to check that $j(E_0)=j(E_1)=j_0$, $\disc(E_0)=j_0^8\cdot(j_0-1728)^{-9}$, and $\disc(E_1)=j_0^{-1}$. We will thus replace $E$ in \Cref{congruence-theorem}  by $E_0$ and give a reinterpretation of $a_\fr{p}(E_0)^2$.

\begin{lemma}
\label{E_0-E_1-isom}
The two elliptic curves $E_0$ and $E_1$ are isomorphic over $L(\sqrt[4]{z_0})$, and they are isomorphic over $L(\sqrt{z_0})$ if and only if $L(\sqrt{z_0})=L(\sqrt[4]{z_0})$.
\end{lemma}
\begin{proof}
Writing $(x_1,y_1)=(x,y+\frac{1}{2}x)$, we obtain that $E_1$ is isomorphic to
\begin{align*}
    y_1^2\eq x_1^3+\frac{1}{4}x_1^2-\frac{1-\sqrt{z_0}}{864}.
\end{align*}
Writing $(x_2,y_2)=(x_1+\frac{1}{12},y_1)$, we obtain that $E_1$ is isomorphic to
\begin{align*}
    E_1':y_2^2\eq x_2^3-\frac{1}{48}x_2+\frac{\sqrt{z_0}}{864}.
\end{align*}
Now, one computes that
\begin{align*}
    \frac{-1/(48z_0^3)}{1/(864z_0^4)}\cdot\frac{\sqrt{z_0}/864}{-1/48}\eq z_0^{3/2}
\end{align*}
and applies \Cref{isom-over-quad} to $E_0$ (over $L(\sqrt{z_0})$) and $E_1'$. The result then follows.
\end{proof}

\begin{corollary}
\label{ap(E0)-ap(E1)}
Let $\fr{p}$ be a prime of $L$ with $\fr{p}\nmid 6$ and $v_\fr{p}(j_0(j_0-1728))=0$. Let $\fr{P}$ be a prime of $L(\sqrt{z_0})$ that lies above $\fr{p}$ and let $p$ be the rational prime that lies below $\fr{p}$. Then,
\begin{align*}
    a_\fr{p}(E_0)^2\;\equiv\;\begin{cases}
        a_\fr{P}(E_1)^2 &\textit{ if $\big(\frac{z_0}{\fr{p}}\big)=1$} \\
        -\big(\frac{-1}{\fr{p}}\big)a_{\fr{P}}(E_1) &\textit{ if $\big(\frac{z_0}{\fr{p}}\big)=-1$}
    \end{cases}\pmod{p}.
\end{align*}
\end{corollary}
\begin{proof}
First, suppose that $\big(\frac{z_0}{\fr{p}}\big)=1$, i.e., $z_0$ is a square in $\F_\fr{p}^\times$. Then, $\F_\fr{p}=\F_\fr{P}$. By \Cref{E_0-E_1-isom}, the reduction of $E_0$ at $\fr{p}$ and the reduction of $E_1$ at $\fr{P}$ are isomorphic over $\F_\fr{p}(\sqrt[4]{z_0})$. By the same argument as in \Cref{same-jinv-square-equal}, one obtains that
\begin{align*}
    a_\fr{p}(E_0)^2\eq a_\fr{P}(E_1)^2.
\end{align*}

Now, suppose that $\big(\frac{z_0}{\fr{p}}\big)=1$, i.e., $z_0$ is not a square in $\F_\fr{p}^\times$. Then, $L\neq L(\sqrt{z_0})$, $\fr{p}=\fr{P}$ is inert in $L(\sqrt{z_0})$, and $\F_\fr{P}=\F_\fr{p}(\sqrt{z_0})$. We first take a digression and discuss when $\sqrt{z_0}$ is a square in~$\F_\fr{P}^\times$. Note that the polynomial $X^2-\sqrt{z_0}$ has a root in $\F_{\fr{P}}$ if and only if $\sqrt{z_0}^{(\N(\fr{P})-1)/2}=1$ in~$\F_{\fr{P}}$, i.e., $z_0^{(\N(\fr{p})^2-1)/4}=1$ in $\F_\fr{p}$. Since $\fr{p}$ is inert in $L(\sqrt{z_0})$, we have $z_0^{(\N(\fr{p})-1)/2}=-1$ in $\F_\fr{p}$. This means that $X^2-\sqrt{z_0}$ has a root in $\F_\fr{P}$ if and only if $(-1)^{(\N(\fr{p})+1)/2}=1$ in $\F_\fr{p}$. That is, $\sqrt{z_0}$ is a square in $\F_\fr{P}^\times$ if and only if $\big(\frac{-1}{\fr{p}}\big)=-1$.

We now return to the proof. Let $E_0'/L(\sqrt{z_0})$ be the base change of $E_0/L$ to $L(\sqrt{z_0})$. Then, by interpreting $a_\fr{p}(E_0)$ and $a_\fr{P}(E_0')$ as traces of Frobeniuses, one obtains that
\begin{align*}
    a_\fr{p}(E_0)^2\;\equiv\; a_\fr{P}(E_0')\pmod{p}.
\end{align*}
Suppose first that $E_0'$ and $E_1$ are already isomorphic over $L(\sqrt{z_0})$. Then, $L(\sqrt{z_0})=L(\sqrt[4]{z_0})$ by \Cref{E_0-E_1-isom}. In particular, $\sqrt{z_0}$ is a square in $\F_\fr{P}^\times$, so $\big(\frac{-1}{\fr{p}}\big)=-1$. Then, 
\begin{align*}
    a_\fr{p}(E_0)^2\;\equiv\;a_\fr{P}(E_0')\;\equiv\;a_\fr{P}(E_1)\;\equiv\;-\big(\tfrac{-1}{\fr{p}}\big)a_{\fr{P}}(E_1)\pmod{p}.
\end{align*}
Now, suppose that $E_0'$ and $E_1$ are not isomorphic over $L(\sqrt{z_0})$. Let $\chi$ denote the quadratic character associated to the quadratic extension $L(\sqrt[4]{z_0})/L(\sqrt{z_0})$, i.e., for a prime $\fr{Q}$ of $L(\sqrt{z_0})$ that is unramified in $L(\sqrt[4]{z_0})/L(\sqrt{z_0})$,
\begin{align*}
    \chi(\fr{Q})\eq\begin{cases}
    1&\text{ if $\fr{Q}$ splits in $L(\sqrt[4]{z_0})$,} \\
    -1&\text{ if $\fr{Q}$ is inert in $L(\sqrt[4]{z_0})$.}
    \end{cases}
\end{align*}
Since $E_0'$ and $E_1$ are isomorphic over $L(\sqrt[4]{z_0})$ but not over $L(\sqrt{z_0})$, by turning to the Galois representations of $E_0'$ and $E_1$, one obtains that
\begin{align*}
    a_\fr{P}(E_0')\eq\chi(\fr{P})a_\fr{P}(E_1).
\end{align*}
Now, by the assumptions on $\fr{p}$ and $j_0$ (and hence $z_0$), $\fr{P}$ splits in $L(\sqrt[4]{z_0})$ if and only if $\sqrt{z_0}$ is a square in $\F_\fr{P}^\times$. Combined with our digression, this implies that $\chi(\fr{P})\eq-(\tfrac{-1}{\fr{p}})$. Hence, 
\begin{align*}
    a_\fr{p}(E_0)^2\;\equiv\;a_\fr{P}(E_0')\;\equiv\;a_\fr{P}(E_1)\;\equiv\;-\big(\tfrac{-1}{\fr{p}}\big)a_{\fr{P}}(E_1)\pmod{p}.
\end{align*}
\end{proof}

To finish this section, we will give a hypergeometric interpretation of $a_\fr{P}(E_1)$. We first prove a lemma on factorials.

\begin{lemma}
\label{congruence-needed-for-ap(E1)}
Let $p\nmid 6$ be a rational prime and let $l\in\Z^+$. Then, for all $0\leq r\leq\frac{p^l-1}{6}$, 
\begin{align*}
    4^{3r}\cdot\frac{(\frac{p^l-1}{2})!}{r!(2r)!(\frac{p^l-1}{2}-3r)!}\;\equiv\;(-1)^{3r}\cdot\frac{(6r)!}{r!(2r)!(3r)!}\pmod{p}.
\end{align*}
\end{lemma}
\begin{proof}
Note that
\begin{align*}
    4^{3r}\cdot\frac{(\frac{p^l-1}{2})!}{(\frac{p^l-1}{2}-3r)!}\eq 2^{3r}\cdot\prod_{i=1}^{3r}(p^l-6r+2i-1)
\end{align*}
and that
\begin{align*}
    (-1)^{3r}\cdot\frac{(6r)!}{(3r)!}\eq(-2)^{3r}\cdot\frac{(6r)!}{\prod_{i=1}^{3r}(2i)}\eq 2^{3r}\cdot\prod_{i=1}^{3r}(-6r+2i-1).
\end{align*}
For each $1\leq i\leq 3r\leq \frac{p^l-1}{2}$, we have $6r-2i+1\leq p^l-1$, so
\begin{align*}
    v_p(-6r+2i-1)\eq v_p(p^l-6r+2i-1)\;=:\;\lambda_i\;<\;l
\end{align*}
and hence
\begin{align*}
    \frac{-6r+2i-1}{p^{\lambda_i}}\;\equiv\;\frac{p^l-6r+2i-1}{p^{\lambda_i}}\pmod{p}.
\end{align*}
In particular,
\begin{align*}
    \prod_{i=1}^{3r}(p^l-6r+2i-1)&\eq p^{\sum_{i=1}^{3r}\lambda_i}\cdot\prod_{i=1}^{3r}\frac{p^l-6r+2i-1}{p^{\lambda_i}} \\
    &\;\equiv\;p^{\sum_{i=1}^{3r}\lambda_i}\cdot\prod_{i=1}^{3r}\frac{-6r+2i-1}{p^{\lambda_i}}\eq\prod_{i=1}^{3r}(-6r+2i-1)\pmod{p^{1+\sum_{i=1}^{3r}\lambda_i}}.
\end{align*}
Write
\begin{align*}
    \lambda\eq v_p\bigg(4^{3r}\cdot\frac{(\frac{p^l-1}{2})!}{(\frac{p^l-1}{2}-3r)!}\bigg)\eq\sum_{i=1}^{3r}\lambda_i\eq v_p\bigg((-1)^{3r}\cdot\frac{(6r)!}{(3r)!}\bigg)
\end{align*}
and
\begin{align*}
    \mu\eq v_p(r!(2r)!).
\end{align*}
Then, $\mu\leq\lambda$ (e.g. by the integrality of $\frac{(6r)!}{r!(2r)!(3r)!}$) and
\begin{align*}
    4^{3r}\cdot\frac{(\frac{p^l-1}{2})!}{r!(2r)!(\frac{p^l-1}{2}-3r)!}\;\equiv\;(-1)^{3r}\cdot\frac{(6r)!}{r!(2r)!(3r)!}\pmod{p^{1+\lambda-\mu}}.
\end{align*}
The result then follows.
\end{proof}

\begin{proposition}
\label{ap(E1)}
Let $\fr{p}$ be a prime of $L$ with $\fr{p}\nmid 6$ and $v_\fr{p}(j_0(j_0-1728))=0$. Let $\fr{P}$ be a prime of $L(\sqrt{z_0})$ that lies above $\fr{p}$. Then,
\begin{align*}
    a_\fr{P}(E_1)\;\equiv\;\pFq{2}{1}{\frac{1}{6},\frac{5}{6}}{1}{\frac{1-\sqrt{z_0}}{2}}_{\lfloor\frac{\N(\fr{P})-1}{6}\rfloor}\pmod{\fr{P}}.
\end{align*}
\end{proposition}
\begin{proof}
Let $p$ be the rational prime that lies below $\fr{p}$ (and $\fr{P}$). Note that
\begin{align*}
    |E_1(\F_\fr{P})|\eq 1+\sum_{x\in\F_\fr{P}}\bigg(\bigg(\frac{x^3+\frac{1}{4}x^2-\frac{1-\sqrt{z_0}}{864}}{\fr{P}}\bigg)+1\bigg),
\end{align*}
so
\begin{align*}
    a_\fr{P}(E_1)\eq \N(\fr{P})+1-|E_1(\F_\fr{P})|
    &\eq-\sum_{x\in\F_\fr{P}}\bigg(\frac{x^3+\frac{1}{4}x^2-\frac{1-\sqrt{z_0}}{864}}{\fr{P}}\bigg) \\
    &\;\equiv\;-\sum_{x\in\F_\fr{P}}\bigg(x^3+\frac{1}{4}x^2-\frac{1-\sqrt{z_0}}{864}\bigg)^{\frac{\N(\fr{P})-1}{2}}\pmod{\fr{P}}.
\end{align*}
Since $\sum_{x\in\F_\fr{P}}x^i\equiv\begin{cases}0&\text{ if }(\N(\fr{P})-1)\nmid i\text{ or }i=0 \\ -1&\text{ if }(\N(\fr{P})-1)\mid i\text{ and }i\neq 0\end{cases}$, we obtain that
\begin{align*}
    a_\fr{P}(E_1)&\;\equiv\;\sum_{\substack{r_3+r_2+r_0=(\N(\fr{P})-1)/2 \\ 3r_3+2r_2=\N(\fr{P})-1}}\binom{\frac{\N(\fr{P})-1}{2}}{r_3, r_2, r_0}\bigg(\frac{1}{4}\bigg)^{r_2}\bigg(-\frac{1-\sqrt{z_0}}{864}\bigg)^{r_0}\pmod{\fr{P}}.
\end{align*}
By the constraints, it is easy to see that $r_3=2r_0$ and $r_2=\frac{\N(\fr{P})-1}{2}-3r_0$. Hence,
\begin{align*}
    a_\fr{P}(E_1)&\;\equiv\;\sum_{r=0}^{\lfloor\frac{\N(\fr{P})-1}{6}\rfloor}4^{3r}\cdot\binom{\frac{\N(\fr{P})-1}{2}}{r, 2r, \frac{\N(\fr{P})-1}{2}-3r}\cdot \bigg(-\frac{1-\sqrt{z_0}}{864}\bigg)^r\pmod{\fr{P}}.
\end{align*}
The result now follows from \Cref{congruence-needed-for-ap(E1)} and the fact that $\frac{(\frac{1}{6})_r(\frac{5}{6})_r}{(r!)^2}\eq 432^{-r}\frac{(6r)!}{r!(2r)!(3r)!}$.
\end{proof}

\section{The case when $\big(\frac{1-1728/j_0}{\fr{p}}\big)=1$}

In this section, we will prove \Cref{congruence-theorem} when $\big(\frac{1-1728/j_0}{\fr{p}}\big)=1$, i.e., when $z_0=1-1728/j_0$ is a square in $\F_\fr{p}^\times$. We first start by showing some congruences of factorials and truncated hypergeometric sums that will be needed for the proof.

\begin{lemma}
\label{2f1-divisible-by-p}
Let $p\nmid 6$ be a rational prime and $l\in\Z^+$. Then, for $\frac{p^l-1}{6}<r\leq p^l-1$,
\begin{align}
\label{2f1-divisible-by-p-eqn}
    \frac{(\frac{1}{6})_r(\frac{5}{6})_r}{(r!)^2}\eq 432^{-r}\frac{(6r)!}{r!(2r)!(3r)!}\;\equiv\;0\pmod{p}.
\end{align}
\end{lemma}
\begin{proof}
Suppose first that $p^l\equiv 1\pmod{6}$. Consider the term $\big(\frac{1}{6}\big)_r/r!$. Note that
\begin{align*}
    v_p\big(\big(\tfrac{1}{6}\big)_r\big)\eq\sum_{s=1}^\infty\hash\big\{0\leq i\leq r-1\;\big|\;\tfrac{1}{6}+i\equiv 0\pmod{p^s}\big\}.
\end{align*}
As $\frac{p^l-1}{6}<r\leq p^l-1$ and $\frac{1}{6}+\frac{p^l-1}{6}\equiv 0\pmod{p^l}$,
\begin{align*}
    \hash\big\{0\leq i\leq r-1\;\big|\;\tfrac{1}{6}+i\equiv 0\pmod{p^l}\big\}\;\geq\;1,
\end{align*}
so
\begin{align*}
    v_p(r!)\eq\sum_{s=1}^{l-1}\bigg\lfloor\frac{r}{p^s}\bigg\rfloor\;\leq\; \sum_{s=1}^{l-1}\hash\big\{0\leq i\leq r-1\;\big|\;\tfrac{1}{6}+i\equiv 0\pmod{p^s}\big\}\;<\;v_p\big(\big(\tfrac{1}{6}\big)_r\big).
\end{align*}
Hence, $v_p\big(\big(\frac{1}{6}\big)_r/r!\big)\geq 1$. Now, suppose that $p^l\equiv 5\pmod{6}$. Then, a similar argument shows that $v_p\big(\big(\frac{5}{6}\big)_r/r!\big)\geq 1$. The result then follows.
\end{proof}

\begin{lemma}
\label{3f2-divisible-by-p}
Let $p\nmid 6$ be a rational prime and $l\in\Z^+$. Then, for $\frac{p^l-1}{6}<r\leq p^l-1$,
\begin{align}
\label{3f2-divisible-by-p-eqn}
    \frac{(\frac{1}{2})_r(\frac{1}{6})_r(\frac{5}{6})_r}{(r!)^3}\eq 1728^{-r}\frac{(6r)!}{(3r)!(r!)^3}\;\equiv\;0\pmod{p}.
\end{align}
\end{lemma}
\begin{proof}
The result follows from the same argument as in \Cref{2f1-divisible-by-p} and that \mbox{$v_p\big(\big(\frac{1}{2}\big)_r/r!\big)\geq 0$}.
\end{proof}

\begin{proposition}
\label{truncated-clausen}
Let $p\nmid 6$ be a rational prime and $l\in\Z^+$. Then,
\begin{align}
\label{truncated-clausen-eqn}
    \pFq{2}{1}{\frac{1}{6},\frac{5}{6}}{1}{t}_{\lfloor(p^l-1)/6\rfloor}^2\;\equiv\;\pFq{3}{2}{\frac{1}{2},\frac{1}{6},\frac{5}{6}}{1,1}{4t(1-t)}_{\lfloor(p^l-1)/6\rfloor}\pmod{p}
\end{align}
as polynomials in $t$.\footnote{One should also see \cite[Lemma~18]{chisholm-deines-long-nebe-swisher} for a mod $p^2$ version of the truncated Clausen's formula with truncation at $p-1$.}
\end{proposition}
\begin{proof}
Throughout this proof, for a power series $\sum_{n=0}^\infty a_nt^n$, we write
\begin{align*}
    \Trun_{m,t}\bigg(\sum_{n=0}^\infty a_nt^n\bigg)\deq\sum_{n=0}^m a_nt^n.
\end{align*}
Note that the usual Clausen's formula gives that
\begin{align*}
    \pFq{2}{1}{\frac{1}{6},\frac{5}{6}}{1}{t}^2\eq\pFq{3}{2}{\frac{1}{2},\frac{1}{6},\frac{5}{6}}{1,1}{4t(1-t)}
\end{align*}
as power series in $t$. Then,
\begin{align*}
    \Trun_{\lfloor(p^l-1)/3\rfloor,t}\bigg(\pFq{2}{1}{\frac{1}{6},\frac{5}{6}}{1}{t}^2\bigg)&\eq\Trun_{\lfloor(p^l-1)/3\rfloor,t}\bigg(\pFq{2}{1}{\frac{1}{6},\frac{5}{6}}{1}{t}_{\lfloor(p^l-1)/3\rfloor}^2\bigg)\\
    &\;\stackrel{(\ref{2f1-divisible-by-p-eqn})}{\equiv}\;\Trun_{\lfloor(p^l-1)/3\rfloor,t}\bigg(\pFq{2}{1}{\frac{1}{6},\frac{5}{6}}{1}{t}_{\lfloor(p^l-1)/6\rfloor}^2\bigg)\\
    &\;\equiv\;\pFq{2}{1}{\frac{1}{6},\frac{5}{6}}{1}{t}_{\lfloor(p^l-1)/6\rfloor}^2\pmod{p},
\end{align*}
and
\begin{align*}
    \Trun_{\lfloor(p^l-1)/3\rfloor,t}\bigg(\pFq{3}{2}{\frac{1}{2},\frac{1}{6},\frac{5}{6}}{1,1}{4t(1-t)}\bigg)&\eq\Trun_{\lfloor(p^l-1)/3\rfloor,t}\bigg(\pFq{3}{2}{\frac{1}{2},\frac{1}{6},\frac{5}{6}}{1,1}{4t(1-t)}_{\lfloor(p^l-1)/3\rfloor}\bigg) \\
    &\;\stackrel{(\ref{3f2-divisible-by-p-eqn})}{\equiv}\;\Trun_{\lfloor(p^l-1)/3\rfloor,t}\bigg(\pFq{3}{2}{\frac{1}{2},\frac{1}{6},\frac{5}{6}}{1,1}{4t(1-t)}_{\lfloor(p^l-1)/6\rfloor}\bigg) \\
    &\;\equiv\;\pFq{3}{2}{\frac{1}{2},\frac{1}{6},\frac{5}{6}}{1,1}{4t(1-t)}_{\lfloor(p^l-1)/6\rfloor}\pmod{p}.
\end{align*}
The result then follows.
\end{proof}

We are now ready to prove \Cref{congruence-theorem} in the case when $(\frac{1-1728/j_0}{\fr{p}})=1$.

\begin{proposition}
Let $\fr{p}$ be a prime of $L$ with $\fr{p}\nmid 6$ and $v_\fr{p}(j_0(j_0-1728))=0$. Suppose that $(\frac{1-1728/j_0}{\fr{p}})=1$. Then,
\begin{align*}
    a_\fr{p}(E_0)^2\;\equiv\;\pFq{3}{2}{\frac{1}{2},\frac{1}{6},\frac{5}{6}}{1,1}{1728/j_0}_{\N(\fr{p})-1}\pmod{\fr{p}}.
\end{align*}
\end{proposition}
\begin{proof}
Write $z_0=1-1728/j_0$ as before. By \Cref{ap(E0)-ap(E1)}, \Cref{ap(E1)}, and that $(\frac{z_0}{\fr{p}})=1$,
\begin{align*}
    a_\fr{p}(E_0)^2\;\equiv\;a_\fr{P}(E_1)^2\;\equiv\;\pFq{2}{1}{\frac{1}{6},\frac{5}{6}}{1}{\frac{1-\sqrt{z_0}}{2}}^2_{\lfloor\frac{\N(\fr{P})-1}{6}\rfloor}\pmod{\fr{P}},
\end{align*}
where $\fr{P}$ is some prime in $L(\sqrt{z_0})$ that lies above $\fr{p}$. It then follows from \Cref{truncated-clausen} with $t=(1-\sqrt{z_0})/2$ and $p^l=\N(\fr{P})$ that
\begin{align*}
    a_\fr{p}(E_0)^2\;\equiv\;\pFq{3}{2}{\frac{1}{2},\frac{1}{6},\frac{5}{6}}{1,1}{1728/j_0}_{\N(\fr{p})-1}\pmod{\fr{P}}.
\end{align*}
The result now follows since both sides are in $L$.
\end{proof}

\section{The case when $\big(\frac{1-1728/j_0}{\fr{p}}\big)=-1$}

In this section, we will prove \Cref{congruence-theorem} in the case when $\big(\frac{1-1728/j_0}{\fr{p}}\big)=-1$, i.e., when $z_0=1-1728/j_0$ is not a square in $\F_\fr{p}^\times$. As before, we will first show more congruences that will be needed for the proof.

\begin{lemma}
\label{pochhammer-over-factorial-congruence}
Let $p$ be a rational prime and let $l\in\Z^+$. For $a\in\Z_p^\times$, let $[a]_0$ be the least nonnegative integer such that $a\equiv[a]_0\pmod{p^l}$ and let $a'=p^{-l}(a+[-a]_0)$.\footnote{One should be cautious that in the theory of hypergeometric functions, the notations $[a]_0$ and $a'$ are usually defined with respect to $p$ but not to $p^l$.} Then, for all $a\in\Z_p^\times$ and $0\leq m\leq p^l-1$,
\begin{align*}
    \frac{(a)_{mp^l}}{(mp^l)!}\;\equiv\;\frac{(a')_m}{m!}\pmod{p}.
\end{align*}
\end{lemma}
\begin{proof}
The congruence is trivially true for $m=0$, so suppose that $1\leq m\leq p^l-1$. Note that
\begin{align*}
    (a)_{mp}\eq\prod_{r=0}^{m-1}\prod_{i=0}^{p^l-1}(a+i+rp^l)&\eq\prod_{r=0}^{m-1}(a+[-a]_0+rp^l)\cdot\prod_{r=0}^{m-1}\prod_{\substack{0\leq i\leq p^l-1 \\ i\neq[-a]_0}}(a+i+rp^l) \\
    &\eq p^{lm}\cdot\prod_{r=0}^{m-1}(a'+r)\cdot\prod_{r=0}^{m-1}\prod_{\substack{0\leq i\leq p^l-1 \\ i\neq[-a]_0}}(a+i+rp^l).
\end{align*}
For $0\leq i\leq p^l-1$ with $i\neq[-a]_0$, we have for all $r\in\Z$,
\begin{align*}
    \lambda_{a,i}\deq v_p([a]_0+i) \eq v_p(a+i)\eq v_p(a+i+rp^l)\;<\;l
\end{align*}
and hence
\begin{align*}
    \frac{[a]_0+i}{p^{\lambda_{a,i}}}\;\equiv\;\frac{a+i+rp^l}{p^{\lambda_{a,i}}}\pmod{p}.
\end{align*}
Thus,
\begin{align*}
    \prod_{\substack{0\leq i\leq p^l-1 \\ i\neq[-a]_0}}\frac{a+i+rp^l}{p^{\lambda_{a,i}}}&\;\equiv\;\prod_{i=0}^{[-a]_0-1}\frac{[a]_0+i+rp^l}{p^{\lambda_{a,i}}}\cdot\prod_{i=[-a]_0+1}^{p^l-1}\frac{([a]_0+i-p^l)+(r+1)p^l}{p^{\lambda_{a,i}}} \\
    &\;\equiv\;p^{-\sum_i\lambda_{a,i}}\cdot(p^l-1)!\pmod{p},
\end{align*}
where $\sum_i$ is taken over $0\leq i\leq p^l-1$ with $i\neq[-a]_0$. Note that
\begin{align*}
    \sum_i\lambda_{a,i}\eq v_p\bigg(\prod_i\big([a]_0+i\big)\bigg)
    &\eq\sum_{s=1}^{l-1}\hash\{0\leq i\leq p^l-1\mid i\neq[-a]_0\text{ and }[a]_0+i\equiv0\pmod{p^s}\} \\
    &\eq\sum_{s=1}^{l-1}(p^{l-s}-1),
\end{align*}
so $\lambda:=\sum_i\lambda_{a,i}$ is independent of $a$. In particular,
\begin{align*}
    p^{-\lambda}\cdot\prod_{\substack{0\leq i\leq p^l-1 \\ i\neq[-a]_0}}(a+i+rp^l)\;\equiv\;p^{-\lambda}\cdot(p^l-1)!\;\equiv\;p^{-\lambda}\cdot\prod_{i=0}^{p^l-2}(1+i+rp^l)\pmod{p}.
\end{align*}
Hence, for $a\in\Z_p^\times$ and $1\leq m\leq p^l-1$,
\begin{align*}
    \frac{(a)_{mp^l}}{(mp^l)!}\eq\frac{\prod\limits_{r=0}^{m-1}(a'+r)\cdot\prod\limits_{r=0}^{m-1}\prod\limits_{\substack{0\leq i\leq p^l-1 \\ l\neq[-a]_0}}(a+i+rp^l)}{\prod\limits_{r=0}^{m-1} (1+r)\cdot\prod\limits_{r=0}^{m-1}\prod\limits_{i=0}^{p^l-2}(1+i+rp)}\;\equiv\;\frac{(a')_m}{m!}\pmod{p}.
\end{align*}
\end{proof}

\begin{lemma}
\label{2f1-p2-and-p}
Let $p\nmid 6$ be a rational prime and let $l\in\Z^+$. Then,
\begin{align}
\label{2f1-p2-and-p-eqn}
    \pFq{2}{1}{\frac{1}{6},\frac{5}{6}}{1}{t}_{p^{2l}-1}\;\equiv\;\pFq{2}{1}{\frac{1}{6},\frac{5}{6}}{1}{t}_{p^l-1}\cdot\pFq{2}{1}{\frac{1}{6},\frac{5}{6}}{1}{t^{p^l}}_{p^l-1}\pmod{p}
\end{align}
as polynomials in $t$.
\end{lemma}
\begin{proof}
Fix the rational prime $p\nmid 6$ and $l\in\Z^+$ and adopt the notations $[a]_0$ and $a'$ in \Cref{pochhammer-over-factorial-congruence} with respect to $p^l$. Then, by \Cref{pochhammer-over-factorial-congruence}, for $a\in\Z_p^\times$ and $0\leq m,n\leq p^l-1$,
\begin{align*}
    \frac{(a)_{mp^l+n}}{(mp^l+n)!}\eq\frac{(a)_{mp^l}(a+mp^l)_n}{(mp^l)!(1+mp^l)_n}\;\equiv\;\frac{(a')_m(a)_n}{m!\cdot n!}\pmod{p}.
\end{align*}
In our case, we see that $((\frac{1}{6})',(\frac{5}{6})')=(\frac{1}{6},\frac{5}{6})$ or $(\frac{5}{6},\frac{1}{6})$ depending on $p^l\pmod{6}$. Hence, for $0\leq m,n\leq p^l-1$
\begin{align*}
    \frac{(\frac{1}{6})_{mp^l+n}(\frac{5}{6})_{mp^l+n}}{((mp^l+n)!)^2}\;\equiv\;\frac{(\frac{1}{6})_m(\frac{5}{6})_m}{(m!)^2}\cdot\frac{(\frac{1}{6})_n(\frac{5}{6})_n}{(n!)^2}\pmod{p}.
\end{align*}
This implies that
\begin{align*}
    \pFq{2}{1}{\frac{1}{6},\frac{5}{6}}{1}{t}_{p^{2l}-1}&\eq\sum_{m,n=0}^{p^l-1}\frac{(\frac{1}{6})_{mp^l+n}(\frac{5}{6})_{mp^l+n}}{((mp^l+n)!)^2}\cdot t^{mp^l+n} \\
    &\;\equiv\;\bigg(\sum_{m=0}^{p^l-1}\frac{(\frac{1}{6})_m(\frac{5}{6})_m}{(m!)^2}\cdot(t^{p^l})^m\bigg)\cdot\bigg(\sum_{n=0}^{p^l-1}\frac{(\frac{1}{6})_n(\frac{5}{6})_n}{(n!)^2}\cdot t^n\bigg) \\
    &\;\equiv\;\pFq{2}{1}{\frac{1}{6},\frac{5}{6}}{1}{t}_{p^l-1}\cdot\pFq{2}{1}{\frac{1}{6},\frac{5}{6}}{1}{t^{p^l}}_{p^l-1}\pmod{p}
\end{align*}
\end{proof}

\begin{proposition}
\label{2f1-z-to-(1-z)}
Let $p\nmid 6$ be a rational prime and let $l\in\Z^+$. Then,
\begin{align}
    \pFq{2}{1}{\frac{1}{6},\frac{5}{6}}{1}{t}_{\lfloor (p^l-1)/6\rfloor}\;\equiv\;(-1)^{(p^l-1)/2}\cdot\pFq{2}{1}{\frac{1}{6},\frac{5}{6}}{1}{1-t}_{\lfloor (p^l-1)/6\rfloor}\pmod{p}
\end{align}
as polynomials in $t$.\footnote{One should also see \cite[Theorem~4.2]{2f1-congruence} and \cite[Lemma~7.2]{babei-roy-swisher-tobin-tu}, both of which imply a mod $p^2$ version of the result with truncation at $p-1$.}
\end{proposition}
\begin{proof}
As before, fix the rational prime $p\nmid 6$ and $l\in\Z^+$ and adopt the notation $[a]_0$ in \Cref{pochhammer-over-factorial-congruence} with respect to $p^l$. Using the similar argument as in \Cref{pochhammer-over-factorial-congruence}, it is easy to show that for $a\in\Z_p^\times$ and $0\leq r\leq[-a]_0$,
\begin{align*}
    (-1)^r\cdot\frac{(a)_r}{r!}\eq\frac{(-a)(-a-1)\cdots(-a-r+1)}{r!}&\;\equiv\;\frac{([-a]_0)([-a]_0-1)\cdots([-a]_0-r+1)}{r!} \\
    &\eq\binom{[-a]_0}{r}\pmod{p}.
\end{align*}
In our case, it is easy to check that
\begin{align*}
    [-\tfrac{1}{6}]_0+[-\tfrac{5}{6}]_0\eq p^l-1\quad\text{ and }\quad\min\big\{[-\tfrac{1}{6}]_0,[-\tfrac{5}{6}]_0\big\}\eq\lfloor\tfrac{p^l-1}{6}\rfloor.
\end{align*}
Let $A=[-\frac{1}{6}]_0$ and $B=[-\frac{5}{6}]_0$. Then,
\begin{align*}
    \pFq{2}{1}{\frac{1}{6},\frac{5}{6}}{1}{t}_{\lfloor\frac{p^l-1}{6}\rfloor}&\eq\sum_{r=0}^{\lfloor (p^l-1)/6\rfloor}\frac{(\frac{1}{6})_r}{r!}\cdot\frac{(\frac{5}{6})_r}{r!}\cdot t^r \\
    &\;\equiv\;\sum_{r=0}^{\min\{A,B\}}\binom{A}{r}\binom{B}{r}\cdot t^r \\
    &\;\equiv\;\text{the constant term of }(1+tX)^A\cdot (1+X^{-1})^B\text{ with respect to }X \pmod{p}.
\end{align*}
Since $\sum_{x\in\F_{p^l}^\times}x^i\equiv\begin{cases}0&\text{ if }(p^l-1)\nmid i\\ -1&\text{ if }(p^l-1)\mid i\end{cases}$, we obtain that
\begin{align*}
    \pFq{2}{1}{\frac{1}{6},\frac{5}{6}}{1}{t}_{\lfloor (p^l-1)/6\rfloor}\eq-\sum_{x\in\F_{p^l}^\times}(1+tx)^A\cdot(1+x^{-1})^B\eq-\sum_{x\in\F_{p^l}\bslash\{0,-1\}}\bigg(\frac{x(1+tx)}{1+x}\bigg)^A
\end{align*}
as elements in $\F_{p^l}[t]$, where we use that $A+B=p^l-1$ and that $(1+x^{-1})^{p^l-1}=1$ for $x\in\F_{p^l}\bslash\{0,-1\}$. Now, let
\begin{align*}
    h(X,t)\eq\frac{X(1+tX)}{1+X}.
\end{align*}
Then,
\begin{align*}
    h\bigg(\frac{1+tX}{t-1},1-t\bigg)\eq\frac{\frac{1+tX}{t-1}\cdot(1-(1+tX))}{1+\frac{1+tX}{t-1}}\eq\frac{(1+tX)(-tX)}{t+tX}\eq-h(X,t).
\end{align*}
Hence, plugging in $t=t_0\in\F_{p^l}\bslash\{0,1\}$, we obtain that
\begin{align*}
    \pFq{2}{1}{\frac{1}{6},\frac{5}{6}}{1}{t_0}_{\lfloor (p^l-1)/6\rfloor}\eq-\sum_{x\in\F_{p^l}\bslash\{0,-1\}}h(x,t_0)^A 
    &\eq-(-1)^A\sum_{x\in\F_{p^l}\bslash\{0,-1\}}h\bigg(\frac{1+t_0x}{t_0-1},1-t_0\bigg)^A \\
    &\eq(-1)^A\cdot\pFq{2}{1}{\frac{1}{6},\frac{5}{6}}{1}{1-t_0}_{\lfloor (p^l-1)/6\rfloor},
\end{align*}
where the last equality follows from the bijectivity of the map $x\mapsto\frac{1+t_0x}{t_0-1}$ on $\F_{p^l}\backslash\{-1\}$. In particular, the two polynomials $\smallpFq{2}{1}{\frac{1}{6},\frac{5}{6}}{1}{t}_{\lfloor (p^l-1)/6\rfloor}$ and $(-1)^A\cdot\smallpFq{2}{1}{\frac{1}{6},\frac{5}{6}}{1}{1-t}_{\lfloor (p^l-1)/6\rfloor}$ agree on $\F_{p^l}\bslash\{0,1\}$, so their difference is divisible by $\sum_{i=0}^{p^l-2}t^i$. As both polynomials are of degree $\leq\lfloor\frac{p^l-1}{6}\rfloor$, they must agree. The result then follows by noting that $(-1)^A=(-1)^{(p^l-1)/2}$.
\end{proof}

\begin{corollary}
\label{2f1-p2-and-p-evaluate}
Let $\fr{p}$ be a prime of $L$ with $\fr{p}\nmid 6$. Let $z_0\in L$ be such that $v_\fr{p}(z_0)=0$ and $\big(\frac{z_0}{\fr{p}}\big)=-1$. Then,
\begin{align}
\label{2f1-p2-and-p-evaluate-eqn}
    \pFq{2}{1}{\frac{1}{6},\frac{5}{6}}{1}{\frac{1-\sqrt{z_0}}{2}}_{\N(\fr{p})^2-1}\;\equiv\;\bigg(\frac{-1}{\fr{p}}\bigg)\cdot\pFq{2}{1}{\frac{1}{6},\frac{5}{6}}{1}{\frac{1-\sqrt{z_0}}{2}}^2_{\N(\fr{p})-1}\pmod{\fr{P}},
\end{align}
where $\fr{P}$ is the unique prime in $L(\sqrt{z_0})$ that lies above $\fr{p}$.
\end{corollary}
\begin{proof}
This simply follows from applying \Cref{2f1-p2-and-p} and \Cref{2f1-z-to-(1-z)} with $t=(1-\sqrt{z_0})/2$ and $p^l=\N(\fr{p})$, where we make use of the fact that $((1-\sqrt{z_0})/2)^{\N(\fr{p})}\equiv(1+\sqrt{z_0})/2\pmod{\fr{P}}$ since $(\frac{z_0}{\fr{p}})=-1$.
\end{proof}

We are now ready to prove \Cref{congruence-theorem} in the case when $(\frac{1-1728/j_0}{\fr{p}})=-1$.

\begin{proposition}
Let $\fr{p}$ be a prime of $L$ with $\fr{p}\nmid 6$ and $v_\fr{p}(j_0(j_0-1728))=0$. Suppose that $(\frac{1-1728/j_0}{\fr{p}})=-1$. Then,
\begin{align*}
    a_\fr{p}(E_0)^2\;\equiv\;-\pFq{3}{2}{\frac{1}{2},\frac{1}{6},\frac{5}{6}}{1,1}{1728/j_0}_{\N(\fr{p})-1}\pmod{\fr{p}}.
\end{align*}
\end{proposition}
\begin{proof}
Write $z_0=1-1728/j_0$ as before. By \Cref{ap(E0)-ap(E1)}, \Cref{ap(E1)}, and that $(\frac{z_0}{\fr{p}})=-1$,
\begin{align*}
    a_\fr{p}(E_0)^2\;\equiv\;-\bigg(\frac{-1}{\fr{p}}\bigg)\cdot a_{\fr{P}}(E_1)\;\equiv\;-\bigg(\frac{-1}{\fr{p}}\bigg)\cdot\pFq{2}{1}{\frac{1}{6},\frac{5}{6}}{1}{\frac{1-\sqrt{z_0}}{2}}_{\lfloor\frac{\N(\fr{P})-1}{6}\rfloor}\pmod{\fr{P}},
\end{align*}
where $\fr{P}$ is the unique prime in $L(\sqrt{z_0})$ that lies above $\fr{p}$. As $\N(\fr{P})=\N(\fr{p})^2$, by \Cref{2f1-divisible-by-p} and \Cref{truncated-clausen} with $t=(1-\sqrt{z_0})/2$ and $p^l=\N(\fr{p})$, and by \Cref{2f1-p2-and-p-evaluate}, we have
\begin{align*}
    \bigg(\frac{-1}{\fr{p}}\bigg)\cdot\pFq{2}{1}{\frac{1}{6},\frac{5}{6}}{1}{\frac{1-\sqrt{z_0}}{2}}_{\lfloor\frac{\N(\fr{P})-1}{6}\rfloor}&\;\stackrel{(\ref{2f1-divisible-by-p-eqn})}{\equiv}\;\bigg(\frac{-1}{\fr{p}}\bigg)\cdot\pFq{2}{1}{\frac{1}{6},\frac{5}{6}}{1}{\frac{1-\sqrt{z_0}}{2}}_{\N(\fr{p})^2-1} \\
    &\;\stackrel{(\ref{2f1-p2-and-p-evaluate-eqn})}{\equiv}\;\pFq{2}{1}{\frac{1}{6},\frac{5}{6}}{1}{\frac{1-\sqrt{z_0}}{2}}^2_{\N(\fr{p})-1} \\
    &\;\stackrel{(\ref{2f1-divisible-by-p-eqn})}{\equiv}\;\pFq{2}{1}{\frac{1}{6},\frac{5}{6}}{1}{\frac{1-\sqrt{z_0}}{2}}^2_{\lfloor\frac{\N(\fr{p})-1}{6}\rfloor} \\
    &\;\stackrel{(\ref{truncated-clausen-eqn})}{\equiv}\;\pFq{3}{2}{\frac{1}{2},\frac{1}{6},\frac{5}{6}}{1,1}{1-z_0}_{\lfloor\frac{\N(\fr{p})-1}{6}\rfloor}\pmod{\fr{P}}.
\end{align*}
Hence, 
\begin{align*}
    a_\fr{p}(E_0)^2\;\equiv\;-\pFq{3}{2}{\frac{1}{2},\frac{1}{6},\frac{5}{6}}{1,1}{1728/j_0}_{\N(\fr{p})-1}\pmod{\fr{P}}.
\end{align*}
The result then follows since both sides are in $L$.
\end{proof}

\bibliographystyle{amsalpha}
\bibliography{ref}

\end{document}